\documentclass[3p]{elsarticle}
\usepackage[x11names,usenames,dvipsnames]{xcolor}
\usepackage{tikz}
\usepackage[T1]{fontenc}

\usepackage{lineno, amsthm}
\usepackage{mathtools}
\usepackage{mathdots,enumerate}
\usepackage{textcomp}
\usepackage{siunitx}
\usepackage{amssymb,amsfonts}
\usepackage{amsmath,paralist,enumitem,mathrsfs,graphicx}
\usepackage{epstopdf}
\usepackage{pgfplots}
\newcommand\norm[1]{\left\lVert#1\right\rVert}
\newcommand\scp[1]{\left\langle#1\right\rangle}
\usetikzlibrary{pgfplots.dateplot}

\newcommand\restr[2]{{% we make the whole thing an ordinary symbol
		\left.\kern-\nulldelimiterspace % automatically resize the bar with \right
		#1 % the function
		\vphantom{\big|} % pretend it's a little taller at normal size
		\right|_{#2} % this is the delimiter
}}
\newcommand{\upperRomannumeral}[1]{\uppercase\expandafter{\romannumeral#1}}
\newcommand{\lowerRomannumeral}[1]{\lowercase\expandafter{\romannumeral#1}}
\usepackage{indentfirst}
\usepackage{epstopdf}
\usepackage[caption=false]{subfig}
\theoremstyle{plain}
\newtheorem{theorem}{Theorem}
\newtheorem{lemma}[theorem]{Lemma}

\newtheorem{corollary}[theorem]{Corollary}

\newtheorem{rem}{Remark}
\theoremstyle{definition}
\newtheorem{example}{Example}

\modulolinenumbers[2]

\makeatletter
\def\ps@pprintTitle{%
	\let\@oddhead\@empty
	\let\@evenhead\@empty
	\def\@oddfoot{\footnotesize\itshape
		\ifx\@empty\@empty
		\else\@journal\fi\hfill\today}%
	\let\@evenfoot\@oddfoot	
}
\makeatother
\begin{document}
	\begin{frontmatter}
		
		\title{Smoothing effect and Derivative formulas for Ornstein--Uhlenbeck processes driven by subordinated cylindrical Brownian noises}
		
		\author{Alessandro Bondi\corref{mycorrespondingauthor}}
		
		\cortext[mycorrespondingauthor]{Classe di Scienze, Scuola Normale Superiore di Pisa, $56126$ Pisa, Italy. Email: {\tt alessandro.bondi@sns.it}} 
		
		\begin{abstract}            
			We investigate the concept of cylindrical Wiener process subordinated to a strictly $\alpha$--stable Lévy process, with $\alpha\in\left(0,1\right)$, in an infinite dimensional, separable Hilbert space, and consider the related stochastic convolution. We then introduce the corresponding Ornstein--Uhlenbeck process, focusing on the regularizing properties of the Markov transition semigroup defined by it. In particular, we provide an explicit, original formula --which is not of Bismut--Elworthy--Li's type-- for the Gateaux derivatives of the functions generated by the operators of the semigroup, as well as an upper bound for the norm of their gradients. In the case $\alpha\in\left(\frac{1}{2},1\right)$, this estimate represents the starting point for studying the Kolmogorov equation in its mild formulation.      
		\end{abstract}
		
	\end{frontmatter}
	
	\section{Introduction}
	The aim of the paper is to analyze the Ornstein--Uhlenbeck processes $Z^x,\,x\in H$, being $H$ an infinite dimensional, separable Hilbert space. They are defined as the $H$--valued, mild solutions of the linear stochastic differential equations
	\[
	dZ^x_t=AZ_t^xdt+\sqrt{Q}\,dW_{L_t},\quad Z^x_0=x\in H,
	\]
	where $A\colon\mathcal{D}\left(A\right)\subset H\to H$ is a linear, selfadjoint, 
	negative definite, unbounded operator, and $Q\colon H\to H$ is a linear, bounded, nonnegative definite operator. By construction, $A$ and $Q$  share a common CONS  of eigenvectors for $H$: it is denoted by $\left(e_n\right)_n$. The main novelty of our work consists in the structure of the noise $W_L$. Intuitively speaking, it can be thought of as
	\[
	W_{L_t}=\sum_{n=1}^{\infty}\beta^n_{L_t}e_n,\quad t\ge0,
	\] 	
	where $\left(\beta^n\right)_n$ is a sequence of independent Brownian motions and $L=\left(L_t\right)_t$ is an independent, strictly $\alpha$--stable subordinator representing the random time change, for $\alpha\in\left(0,1\right)$. Therefore $W_L$ is nothing else than a subordinated cylindrical Wiener process, even if, of course, the convergence of the series needs to be formally investigated.
	
	In literature the canonical case is the Gaussian one, which involves a cylindrical Wiener process $W_t=\sum_{n=1}^{\infty}\beta^n_{t}e_n,\, t\ge0.$ There is a well--established theory concerning this setting, and we may refer to the book \cite{DP2} for an extensive collection of results on the subject. Another important framework is the one proposed by \cite{PZ}, where the authors deal with a cylindrical, $\alpha$--stable Lévy process $Z_t=\sum_{n=1}^{\infty}\zeta^n_{t}e_n,\, t\ge0$: here $\left(\zeta^n\right)_n$ are independent, real--valued, symmetric $\alpha$--stable Lévy processes, for $\alpha\in\left(0,2\right)$. Despite the interesting generalization offered by this approach, the structure of the noise could be questionable in some applications, especially in physics. In fact, fixing $t>0$ and $N\in\mathbb{N}$, the corresponding Galerkin projection of $Z_t$ has characteristic function
	\[
	\mathbb{E}\left[e^{i\left\langle h,\sum_{n=1}^{N}\zeta^n_te_n\right\rangle}\right]=e^{-t\gamma^\alpha \sum_{n=1}^N\left|\left\langle h,e_n\right\rangle\right|^\alpha},\quad h\in H,
	\]
	for some constant $\gamma>0$. Therefore with respect to the Brownian case we lose the isotropy, that is, the rotational stability of the noise, which is a property as desirable as realistic for a random perturbation. 
	
	Motivated by this argument, it is worth studying the results contained in the aforementioned works also for the subordinated process $W_L$, since its Galerkin projections are $2\alpha$--stable, isotropic Lévy processes, as we shall discuss in Section \ref{sec1}. With this purpose in mind, the present paper just focuses on the linear case, i.e., the Ornstein--Ulenbeck (henceforth abbreviated as \emph{OU}) one. A number of complications arises from the approach that we suggest, the most evident being the lack of independence of the processes $\left(\beta^n_L\right)_n$, which in general makes the techniques used in the other cases unfeasible. Nevertheless, the structure of the noise allows to construct the objects of our interest and to carry out our arguments  with the intuition that, conditioning on the $\sigma$--algebra generated by the subordinator $L$, we are dealing with time--shifted Brownian motions.
	
	The paper is structured as follows. In Section \ref{sec1} we carefully describe the theoretical framework of our analysis and suggest a natural procedure --essentially relying on \emph{Markov's inequality}-- to construct both the subordinated cylindrical Wiener process $W_L$, or, more precisely, $\sqrt{Q}W_L$, which in general takes values in a Hilbert space bigger than $H$, and the stochastic convolution $\tilde{Z}_{A,Q}$, which is a $H$--valued random process instead. \\In Section \ref{sec2} we are concerned with the smoothing effect of the Markov transition semigroup $R=\left(R_t\right)_{t\ge0}$ associated with $\left(Z^x\right)_{x\in H}$, defined by 
	\[
	R_t\phi\left(x\right)\coloneqq\mathbb{E}\left[\phi\left(Z^x_t\right)\right],\quad x\in H,\, \phi\in \mathcal{B}_b\left(H\right),\,t\ge0.
	\]
	We first study the finite--dimensional case, starting with a deterministic time change (see Theorem \ref{determ}) and subsequently recovering the random time shift in Theorem \ref{detgrad}. This way to proceed is customary while working with subordinated Brownian motions (see, e.g., \cite{KU,Z}). Taking advantage of the linear structure of our model, we are able to get a derivation formula for $R_t\phi$ (see Equation \eqref{no_bel}) with a density argument, shunning an application of the Bismut--Elworthy--Li's type formula provided by \cite{Z}. This is a remarkable fact, also because it is consistent with the Gaussian framework, where it is preferable to use the \emph{Bismut--Elworthy--Li's formula} only in the nonlinear case. Then in Theorem \ref{main} we pass to the general, infinite--dimensional setting under suitable assumptions. A subtle difference between the finite-- and infinite--dimensional cases is that in the former we get an expression for the Gateaux derivative of $R_t\phi$ for every $\phi\in \mathcal{B}_b\left(H\right)$, whereas in the latter such a formula (see Equation \eqref{no_BEL}) holds true only for $\phi\in C_b\left(H\right)$. In addition, in Corollary \ref{corol} we provide a gradient estimate that, for $\alpha\in\left(\frac{1}{2},1\right)$, represents the starting point for the analysis of the Kolmogorov equation in its mild form with fixed--point arguments, analysis which will be the topic of a future research.\\
	Each of the previous two sections is closed by an example which studies a concrete framework, namely $H=L^2_0\left(\mathbb{T}^d\right)$, with $\mathbb{T}^d=\mathbb{R}^d/\mathbb{Z}^d$ being the $d$--dimensional torus. Herein we discuss the hypotheses required by the several theorems of the paper, with explicit computations that offer a parallel with the corresponding, well--known results of the Gaussian setting.  
	\section{Subordinated Cylindrical Wiener Process and Stochastic Convolution}\label{sec1}
	Let $H$ be a separable Hilbert space and $\left(e_n\right)_n$ be a complete orthonormal system. We consider a complete probability space $\left(\Omega,\mathcal{F},\mathbb{P}\right)$ and introduce a sequence of independent Brownian motions $\left(\beta^n\right)_n$ on it. Let $L=\left(L_t\right)_t$ be a strictly $\alpha$--stable subordinator, i.e., an increasing L\'evy process where the distribution of $L_1\sim \mu$ is characterized by
	\begin{equation}\label{forg}
	\hat{\mu}\left(u\right)=\exp\left\{-\bar{c}\left|u\right|^\alpha\left(1-i\tan\frac{\pi\alpha}{2}\,\text{sign}\left(u\right)\right)\right\},\quad u\in\mathbb{R},
	\end{equation}
	with $\bar{c}>0, \,\alpha\in\left(0,1\right)$. The Laplace transform of $\mu$ is given by 
	\begin{equation}\label{1}
	L_\mu\left(u\right)=\mathbb{E}\left[e^{-uL_1}\right]=e^{-c'u^\alpha},\quad u\ge0,
	\end{equation}
	where $c'$ is a constant depending on $\bar{c}$ (for an expression of $c'$ we refer to \cite[Example $24.12$]{Sato}, but it is of no use in our work). Let us introduce the subordinated Brownian motions  $\left(\beta^n_{L_t}\right)_t, n\in\mathbb{N}$: assuming $L$ to be independent from  $\left(\beta^n\right)_n$, \cite[Theorem $30.1$]{Sato} implies that $\left(\beta^n_L\right)_n$ are real--valued L\'evy processes.
	
	 Denoting by $\mathcal{N}$\,the family of $\mathcal{F}$--negligible sets, we introduce the augmented $\sigma$--algebra  $\mathcal{F}^L\coloneqq\sigma\left(\mathcal{F}^L_0\cup\mathcal{N}\right)$, where $\mathcal{F}^L_0$ is the natural $\sigma$--algebra generated by the subordinator. Analogously, we consider the augmented $\sigma$--algebras $\mathcal{F}^{\beta^n}$ generated by the Brownian motions. Thanks to the hypotheses of independence that we have assumed on the processes, we have that $\mathcal{F}^L,\,\left(\mathcal{F}^{\beta^n}\right)_n$ are mutually independent. In our context, it is natural to deal with different filtrations. Specifically, for every $n\in\mathbb{N}$ let $\mathbb{F}^n=\left(\mathcal{F}^{\beta^n}_t\right)_t$ be the minimal augmented filtration generated by $\beta^n$, that is, $\mathcal{F}^{\beta^n}_t\coloneqq\sigma\left(\left({\mathcal{F}_0^{\beta^n}}\right)_t\cup\mathcal{N}\right)$ for every $t\ge0$, where $\left(\left({\mathcal{F}_0^{\beta^n}}\right)_t\right)_t$ is the natural filtration of the process. According to \cite[Theorem \upperRomannumeral{1}$.31$]{Protter}, $\mathbb{F}^n$ satisfies the usual hypotheses.
	 Then we construct a complete filtration associated with the subordinated Brownian motions. It is denoted  by $\mathbb{F}_L=\left(\mathcal{F}_t\right)_t$, where we define
	 \[
	 	\mathcal{F}_t\coloneqq\sigma\left(\bigcup_{n\in\mathbb{N}}\mathcal{F}^{\beta_L^n}_{t}\right),\quad t\ge0,
	 \]
	 with $\mathbb{F}^n_L=\left(\mathcal{F}^{\beta_L^n}_t\right)_t$ being the minimal augmented filtration associated with $\beta^n_L$. 
	 \begin{rem}
	 	In the finite dimensional case, we denote by $W_L^N=\left(W^N_{L_t}\right)_t$ the subordinated, $\mathbb{R}^N$--valued Brownian motion, meaning that
	 	\[
	 		W^N_{L_t}=\left[\begin{matrix}
	 		\beta^1_{L_t}&\cdots&\beta^N_{L_t}
	 		\end{matrix}\right]^T,\quad t\ge0.
	 		\]
	 	By \cite[Theorem $30.1$]{Sato}, $W^N_L$ is an $\mathbb{R}^N$--valued Lévy process, and it is easy to verify that its minimal augmented filtration $\left(\mathcal{F}^{W_L^N}_t\right)_t$ coincides with $\mathbb{F}_L$. This fact shows that the construction that we have carried out for $\mathbb{F}_L$ is natural.
	 	
	 	Using the notation we have just introduced, in the general case the $\sigma$--algebras constituting $\mathbb{F}_L$ can be expressed as follows: $$\mathcal{F}_t=\sigma\left(\bigcup_{N\in\mathbb{N}}\mathcal{F}^{W_L^N}_t\right),\quad t\ge0.$$
	 \end{rem}
 	\subsection{Subordinated Cylindrical Wiener Process}
	The aim of this section is to give a rigorous meaning to the formal notation $W_{L_t}=\sum_{n=1}^\infty \beta^n_{L_t} e_n, t>0$. 
	
	First, fix $h\in H,\, t>0$ and notice that the series $\sum_{n=1}^{\infty}\beta^n_{L_t}\left\langle h,e_n\right\rangle$ converges in distribution. Indeed, even if the random variables $\left(\beta^n_{L_t}\right)_{n\in\mathbb{N}}$ are not independent due to the presence of the subordinator, we can still exploit the mutual independence of the $\sigma$--algebras $\left(\mathcal{F}^{\beta^n}\right)_n$ by conditioning with respect to $\mathcal{F}^L$, which in turn is independent from the previous ones.  In order to do so, we use the law of total expectation together with \eqref{1} to get, for every $u\in\mathbb{R}$,
	\begin{align}\label{eq1}
	\mathbb{E}\left[\exp\left\{iu\sum_{n=1}^N\beta^n_{L_t}\left\langle h,e_n\right\rangle\right\}\right]
	&\!=\mathbb{E}\left[\mathbb{E}\left[\restr{\exp\left\{iu\!\sum_{n=1}^N\beta^n_{r}\left\langle h,e_n\right\rangle\right\}}{r=L_t}\Bigg|\mathcal{F}^L\right]\right]	\notag\\
	&=
	\mathbb{E}\left[\restr{\mathbb{E}\left[\exp\left\{iu\!\sum_{n=1}^N\beta^n_{r}\left\langle h,e_n\right\rangle\right\}
	\right]}{r=L_t}\right]\!\!
	=\mathbb{E}\left[\prod_{n=1}^{N}\exp\left\{-\frac{1}{2}L_t\left|u\right|^2\left|\left\langle h,e_n\right\rangle\right|^2\right\}\right] \notag\\%=\mathbb{E}\left[\exp\left\{-\frac{1}{2}L_t\left|u\right|^2\sum_{n=1}^N\left|\left\langle h,e_n\right\rangle\right|^2\right\}\right] 
	&=\exp\left\{-tc'\frac{1}{2^\alpha}\left|u\right|^{2\alpha}\left(\sum_{n=1}^N\left|\left\langle h,e_n\right\rangle\right|^2\right)^\alpha\right\}\underset{N\to\infty}{\longrightarrow}
	\exp\left\{-t\frac{c'}{2^\alpha}\norm{h}^{2\alpha}_H\left|u\right|^{2\alpha}\right\}.
	\end{align}
	Hence applying L\'evy's continuity theorem we see that the series $\sum_{n=1}^{\infty}\beta^n_{L_t}\left\langle h,e_n\right\rangle$ converges in distribution to a symmetric, $2\alpha$--stable random variable. Moreover, for every $n\in\mathbb{N}$, choosing $h=e_n$ and $N>n$ the computations in \eqref{eq1} provide the distribution of the Lévy process $\beta^n_L$, namely
	\begin{equation}\label{eq0}
	\mathbb{E}\left[e^{iu\beta^n_{L_t}}\right]=\exp\left\{-t\frac{c'}{2^\alpha}\left|u\right|^{2\alpha}\right\},\quad u\in\mathbb{R}, \,\text{ for any } t>0.
	\end{equation}
 The process $W_L=\left(W_{L_t}\right)_t$ is a \emph{subordinated cylindrical Wiener process}, but we might also call it   \emph{cylindrical, $2\alpha$--stable isotropic process}. In fact, for every $N\in\mathbb{N}$ and $t>0$, if we denote by $\pi_N$ the projection onto the first $N$ Fourier components and by $H_N$  its range, an argument analogous to the one in \eqref{eq1} yields: 
	\begin{equation*}
	\mathbb{E}\left[\exp\left\{i\left\langle z,\sum_{n=1}^N\beta^n_{L_t}e_n\right\rangle\right\}\right]
	=\exp\left\{-t\frac{c'}{2^\alpha}\left(\sum_{n=1}^N\left|\left\langle z,e_n\right\rangle\right|^2\right)^\alpha\right\},\quad z\in{H}.
	\end{equation*}
	Hence canonically identifying $H_N$ with $\mathbb{R}^N$, the  Galerkin projection $\left(\sum_{n=1}^N \beta^n_{L_t}e_n\right)_t$ can be read as an {$\mathbb{R}^N$--valued}, $2\alpha$--stable, isotropic L\'evy process. 
	
	Secondly, we consider a linear, bounded, nonnegative definite operator $Q:H\to H$ such that $e_n$ is one of its eigenvectors corresponding to the eigenvalue $\sigma_n^2\ge0$, $n\in\mathbb{N}$. We study the convergence in probability --on an appropriate space-- of the series: 
	\[
	\sqrt{Q}W_{L_t}=\sum_{n=1}^\infty\sigma_n\beta^n_{L_t}e_n, \quad t>0.
	\]
	 Let us introduce a bounded sequence $\left(\rho_n\right)_n$ of strictly positive numbers such that $\sum_{n=1}^{\infty}\rho_n^{2r}\sigma_n^{2r}<\infty$ for some $r\in\left(0,\alpha\right)$, and consider the corresponding Hilbert space $\left(V,\left\langle\cdot,\cdot\right\rangle_V\right)$, where
	\begin{equation}\label{Gelfand}
	V\coloneqq\left\{h\in H : \sum_{n=1}^{\infty}\rho_n^{-2}\left|\left\langle h,e_n\right\rangle\right|^2<\infty\right\}\quad\text{and}\quad \left\langle v,w\right\rangle_V\coloneqq\sum_{n=1}^{\infty}\rho_n^{-2}\left\langle v,e_n\right\rangle\left\langle w,e_n\right\rangle,\quad v,w\in V.
	\end{equation}
	Evidently $V\subset H$ with dense and continuous embedding, therefore using the concept of \emph{Gelfand triple} we can think a generic $h\in H$ as an object in $V'$, namely
	\[
	\left\langle h,v\right\rangle_{V',V}=\sum_{n=1}^\infty\left\langle h,e_n\right\rangle \left\langle v,e_n\right\rangle,\quad v\in V.
	\]
	Noticing that $\scp{h,\cdot}_{V',V}=\scp{\tilde{v},\cdot}_{V}$, where $\tilde{v}\coloneqq\sum_{n=1}^\infty \rho_n^{2}\scp{h,e_n}e_n\in V$, we can apply \emph{Riesz representation theorem} to get $\norm{h}^2_{V'}=\sum_{n=1}^\infty\rho_n^{2}\left|\scp{h,e_n}\right|^2.$ Now we fix $t>0$ and show that  $\left(\sum_{n=1}^N\sigma_n\beta^n_{L_t}e_n\right)_N\subset V'$ is a Cauchy sequence in probability. Indeed, applying \emph{Markov's inequality }and using the fact that the function $\phi\left(x\right)=x^r,\,x\ge0$, is subadditive and strictly increasing as $0<r<\alpha<1$,  for every $\epsilon>0$ we get:
	\begin{align*}
	\mathbb{P}&\left(\norm{\sum_{n=p}^q\sigma_n\beta^n_{L_t}e_n}_{V'}>\epsilon\right)
	\le\mathbb{P}\left(\phi\left(\norm{\sum_{n=p}^q\sigma_n\beta^n_{L_t}e_n}_{V'}^2\right)>\phi\left(\epsilon^2\right)\right)
	\le\frac{1}{\epsilon^{2r}}\mathbb{E}\left[\phi\left(\norm{\sum_{n=p}^q\sigma_n\beta^n_{L_t}e_n}_{V'}^2\right)\right]\\
	&=\!\epsilon^{-2r}\mathbb{E}\left[\phi\left(\sum_{n=p}^q\sigma_n^2\rho_n^2\left|\beta^n_{L_t}\right|^2\right)\right]\!
	\le\epsilon^{-2r}\sum_{n=p}^q\mathbb{E}\left[\left(\sigma_n^{2r}\rho_n^{2r}\left|\beta^n_{L_t}\right|^{2r}\right)\right]\!
	=\!\epsilon^{-2r}\mathbb{E}\left[\left|\beta^1_{L_t}\right|^{2r}\right]\left(\sum_{n=p}^q\sigma_n^{2r}\rho_n^{2r}\right)\!\underset{p,q\to\infty}{\longrightarrow}0,
	\end{align*}
	where we use that by construction $\beta^n_{L_t}\sim \beta^1_{L_t}, n\in\mathbb{N},$ and that by \eqref{eq0} they all generate a $2\alpha$--stable distribution, which has finite moment of order $2r$ (see also Remark \ref{r1}). By completeness, we can conclude the existence of an a.s. unique, $V'$--valued random variable $\sqrt{Q}W_{L_t}$ such that 
	\[
	\sqrt{Q}W_{L_t}=\mathbb{P}-\lim_{N\to\infty} \sum_{n=1}^N\sigma_n\beta^n_{L_t}e_n \quad \text{in }V'.
	\] 
	Actually such a convergence in probability is true also in the $\mathbb{P}-$a.s. sense, as the following, easy and general lemma proves.
	\begin{lemma}\label{itsa.s.}
		Let $\left(X^n\right)_n$ be a sequence of real--valued random variables defined on a probability space $\left(\Omega,\mathcal{F},\mathbb{P}\right)$ and $H$ be a separable Hilbert space admitting $\left(e_n\right)_n$ as CONS. If $\sum_{n=1}^{\infty} X^ne_n$ converges in probability, then it converges $\mathbb{P}-$a.s.
	\end{lemma}
\begin{proof}
	Let $S\coloneqq\mathbb{P}-\lim_{N\to \infty}\sum_{n=1}^NX^ne_n\colon\Omega\to H$.  Obviously 
	\begin{equation}\label{mi1}
	S\left(\omega\right)=
	\sum_{n=1}^\infty \left\langle S\left(\omega\right),e_n\right\rangle e_n
	=H-\!\!\lim_{N\to \infty}\sum_{n=1}^N \left\langle S\left(\omega\right),e_n\right\rangle e_n,\quad \omega\in\Omega.
	\end{equation}
	Convergence in measure implies a.s. convergence along a subsequence, hence we have
	\[
	S\left(\omega\right)
	=H-\!\!\lim_{k\to \infty}\sum_{n=1}^{N_k} X^n\left(\omega\right)e_n\quad \text{for }\mathbb{P}-\text{a.e. }\omega\in\Omega.
	\]
	Therefore, for $\mathbb{P}-$a.e. $\omega\in\Omega$, we see that the Fourier components of $S$ are 
	\[
	\left\langle S\left(\omega\right),e_{\bar{n}}\right\rangle		=\lim_{k\to \infty}\left\langle\sum_{n=1}^{N_k}X^n\left(\omega\right)e_n,e_{\bar{n}}\right\rangle=X^{\bar{n}}\left(\omega\right)\quad \text{for every }\bar{n}\in\mathbb{N}.
	\]
	Substituting in \eqref{mi1} we conclude
	\[
	S\left(\omega\right)
	=H-\!\!\lim_{N\to \infty}\sum_{n=1}^N X^n\left(\omega\right) e_n \quad \text{for } \mathbb{P}-\text{a.e. }\omega\in\Omega,
	\]
	as we stated.
\end{proof}
Going back to $\sqrt{Q}W_{L_t}$, since $\left(\rho_ne_n\right)_n$ is a CONS for the Hilbert space $V$, Lemma \ref{itsa.s.} allows to write
\[
	\sqrt{Q}W_{L_t}=\lim_{N\to\infty}\left\langle \sum_{n=1}^N\sigma_n\beta^n_{L_t}e_n,\cdot\right\rangle_{V',V}
	=\lim_{N\to\infty}\sum_{n=1}^N\rho_n\sigma_n\beta^n_{L_t}\left\langle \left(\rho_ne_n\right),\cdot\right\rangle_{V}\quad\mathbb{P}-\text{a.s.}
\]
It then follows that $\left\langle\sqrt{Q}W_{L_t},v\right\rangle_{V',V}=\lim_{N\to \infty}\sum_{n=1}^N\sigma_n\beta^n_{L_t}\left\langle v,e_n\right\rangle$ for every $v\in V,\,\mathbb{P}-$a.s. Combining this with \eqref{eq1}, we can see that $\left\langle\sqrt{Q}W_{L_t},v\right\rangle_{V',V}$ has a symmetric, $2\alpha$--stable distribution.
We collect the  previous results in the next theorem.
\begin{theorem}
	\begin{enumerate}
	\item Given $h\in H$ and $t>0$, the series $\sum_{n=1}^{\infty}\beta^n_{L_t}\left\langle h,e_n\right\rangle$ converges in distribution to a real--valued, symmetric, $2\alpha$--stable random variable $X_t$ whose characteristic function is
	\[
		\mathbb{E}\left[e^{iuX_t}\right]=\exp\left\{-t\frac{c'}{2^\alpha}\norm{h}^{2\alpha}_H\left|u\right|^{2\alpha}\right\},\quad u\in\mathbb{R}.
	\]
	\item	Consider a linear, bounded, nonnegative definite operator $Q:H\to H$ such that $\left(e_n\right)_n$ is a basis of its eigenvectors corresponding to the eigenvalues $\left(\sigma_n^2\right)_{n}\left(\subset \mathbb{R}_+\right)$.	Let $\left(\rho_n\right)_n$ be a bounded sequence of strictly positive weights such that $\sum_{n=1}^{\infty}\rho_n^{2r}\sigma_n^{2r}<\infty$ for some $0<r<{\alpha}$. Then the corresponding Hilbert space $\left(V,\left\langle\cdot,\cdot\right\rangle_V\right)$ defined in \eqref{Gelfand} is continuously embedded with density in $H$ and, for every $t>0$, the random variable $\sqrt{Q}W_{L_t}\colon\Omega\to V'$ is defined as
	\[
	\sqrt{Q}W_{L_t}\coloneqq\lim_{N\to\infty} \sum_{n=1}^N\sigma_n\beta^n_{L_t}e_n\quad \mathbb{P}-\text{a.s.}
	\] 
In particular, for every $v\in V$,
\[
	\sum_{n=1}^{N}\beta^n_{L_t}\left\langle\sqrt{Q}v,e_n\right\rangle \underset{N\to\infty}{\longrightarrow}\left\langle\sqrt{Q}W_{L_t},v\right\rangle_{V',V}\quad \mathbb{P}-\text{a.s.}
\]
\end{enumerate}
\end{theorem}
	\begin{rem}\label{r1}
	We can state the finiteness of the moment of order $2r$ of the random variable $\beta^1_{L_t}$ without explicitly knowing its distribution, i.e., without using \eqref{eq0}. In fact, we can proceed as follows:
	\[
	\mathbb{E}\left[\left|\beta^1_{L_t}\right|^{2r}\right]=\mathbb{E}\left[\mathbb{E}\left[\left|\beta^1_{L_t}\right|^{2r}\Bigr|\mathcal{F}_L\right]\right]=\frac{2^{r}}{\sqrt{\pi}}\Gamma\left(\frac{2r+1}{2}\right)\mathbb{E}\left[L_t^r\right]<\infty,
	\]
	since we are dealing with $0<r<\alpha.$ For the second equality we refer to \cite[Equation $\left(17\right)$]{Wink}.
\end{rem}
	\subsection{Stochastic Convolution}
	 Let $A:\mathcal{D}\left(A\right)\subset H\to H$ be a linear, selfadjoint, 
	 negative definite, unbounded operator %with compact resolvent commuting with $Q$,
	 that shares with $Q$ a common basis of eigenvectors $\left(e_n\right)_n$. We denote by $\left(-\lambda_n\right)_n$, with $0<\lambda_1\le\lambda_2 \le \dots\le \lambda_n\le \cdots$ the corresponding eigenvalues, i.e., $Ae_n=-\lambda_ne_n,\,n\in\mathbb{N}$. Recalling that $\alpha\in\left(0,1\right)$ has been fixed at the beginning of Section \ref{sec1}, it is convenient to introduce the shorthand notation $X\sim \text{stable}\left(\alpha,\beta,\gamma,\delta\right)$ to denote a random variable $X$ with characteristic function given by
	\[
	\mathbb{E}\left[e^{iuX}\right]=\exp\left\{-\gamma^\alpha\left|u\right|^\alpha\left(1-i\beta\tan\frac{\pi\alpha}{2}\text{sign}\left(u\right)\right)+i\delta u\right\},\quad u\in\mathbb{R}.
	\]
	where $\left|\beta\right|\le1,\gamma>0$ and $\delta\in\mathbb{R}$. 
	Hence by \eqref{eq0}, for every $n\in\mathbb{N}$ the L\'evy process $\beta^n_L$ has random variables distributed as 
	\[
		\beta^n_{L_t}\sim\text{stable}\left(2\alpha,0,\left(t\frac{c'}{2^\alpha}\right)^{{1}/\left({2\alpha}\right)},0\right),\quad t>0.
	\]
	 We denote by $U^n=\left(U_t^n\right)_{t\ge0}$ the OU--process
	$
	U_t^n\coloneqq\int_{0}^t e^{-\lambda_n\left(t-s\right)}\sigma_n\,d\beta^n_{L_s},\, t\ge0:
	$
	this is the unique (up to evanescence) solution of the one dimensional stochastic differential equation
	\begin{equation}\label{tea}
	dU^n_t=-\lambda_nU^n_tdt+\sigma_n\,d\beta^n_{L_t},\quad U^n_0=0.
	\end{equation}
	The processes $\left(U^n\right)_n$ are c\`adl\`ag and adapted to the filtration $\mathbb{F}_L$, and direct computations (see, e.g., \cite[Proposition $3.2$]{BB}) show that $U_t^n\sim \text{stable}(2\alpha,0, \gamma_n\left(t\right),0)$, where 
	\[
	\gamma_n\left(t\right)\coloneqq \left(\frac{c'}{2^\alpha}\right)^{{1}/\left({2\alpha}\right)}\!\!\left(\int_{0}^t\!\!e^{-2\alpha\lambda_n\left(t-s\right)}\sigma_n^{2\alpha}\,ds\right)^{{1}/\left({2\alpha}\right)}\!\!\!
	=\sigma_n\left(\frac{c'}{2^{\alpha+1}\alpha}\right)^{{1}/\left({2\alpha}\right)}\!\!\left(\frac{1-e^{-2\alpha\lambda_nt}}{\lambda_n}\right)^{{1}/\left({2\alpha}\right)}\!\!,\quad t>0,\,n\in\mathbb{N}.
	\]
	We are now in position to construct the \emph{stochastic convolution} and the corresponding OU--process.
	\begin{theorem}
		Assume that
		\begin{equation}\label{st_con}\tag{\lowerRomannumeral{1}}
		\sum_{n=1}^\infty\dfrac{\sigma_n^{2r}}{\lambda_n^{{r}/{\alpha}}}<\infty \quad \text{for some }r\in\left(0,\alpha\right).
		\end{equation}
	Then,\,for\,all\,$t\!>\!0$,\,the\,series\,$\sum_{n=1}^\infty\!U^n_te_n$ converges $\mathbb{P}-$a.s.\,to\,a random variable\,$\tilde{Z}_{A,Q}\!\left(t\right)=\!\int_{0}^{t}e^{\left(t-s\right)A}\sqrt{Q}\,dW_{L_s}$. The resulting process $\tilde{Z}_{A,Q}=\left(\tilde{Z}_{A,Q}\left(t\right)\right)_t$ is $\mathbb{F}_L$--adapted and is called \emph{stochastic convolution}. 
	
	The corresponding OU--process starting at $x\in H$, denoted by 	$Z^x=\left(Z_t^x\right)_t$ and defined by
	\[
	Z_t^x\coloneqq e^{tA}x+\int_{0}^{t}e^{\left(t-s\right)A}\sqrt{Q}\,dW_{L_s}=e^{tA}x+\tilde{Z}_{A,Q}\left(t\right),\quad t\ge 0,
	\]
	is $\mathbb{F}_L$--adapted and Markovian with homogeneity in time.
\end{theorem}
\begin{proof}
	Fix $t>0$.
	Thanks to the preceding discussion, we know that $U^n_t\sim\gamma_n\left(t\right)X,\,n\in\mathbb{N}$, where $X$ is a random variable such that $X\sim\text{stable}\left(2\alpha,0,1,0\right)$. Then an application of Markov's inequality entails:
	\begin{multline*}
	\mathbb{P}\left(\norm{\sum_{n=p}^qU^n_te_n}_{H}\!\!\!>\epsilon\right)\!\!
	\le\epsilon^{-2r}\mathbb{E}\left[\phi\left(\norm{\sum_{n=p}^qU^n_te_n}_{H}^2\right)\right]\!\!
	=\epsilon^{-2r}\mathbb{E}\left[\phi\left(\sum_{n=p}^q\left|U^n_t\right|^2\right)\right]\!\!
	\le\epsilon^{-2r}\sum_{n=p}^q\mathbb{E}\left[\left(\left|U^n_t\right|^{2r}\right)\right]\\
	=\epsilon^{-2r}\mathbb{E}\left[\left|X\right|^{2r}\right]\left(\frac{c'}{2^{\alpha+1}\alpha}\right)^{r/\alpha}\left(\sum_{n=p}^q\frac{\sigma_n^{2r}}{\lambda_n^{r/\alpha}}\left(1-e^{-2\alpha\lambda_nt}\right)^{r/\alpha}\right)
	\le c\left(\epsilon\right)\left(\sum_{n=p}^q\frac{\sigma_n^{2r}}{\lambda_n^{r/\alpha}}\right)
	\underset{p,q\to\infty}{\longrightarrow}0,\quad \epsilon>0,
	\end{multline*} 
	with $c\left(\epsilon\right)\coloneqq \epsilon^{-2r}\mathbb{E}\left[\left|X\right|^{2r}\right]\left(\dfrac{c'}{2^{\alpha+1}\alpha}\right)^{{r}/{\alpha}}$ and $\phi\left(x\right)=x^r$, as above. Therefore the series converges in probability: 
	\[
	\tilde{Z}_{A,Q}\left(t\right)=\int_{0}^{t}e^{\left(t-s\right)A}\sqrt{Q}\,dW_{L_s}\coloneqq\mathbb{P}-\lim_{N\to \infty}\sum_{n=1}^N U^n_te_n.
	\]
	An application of Lemma \ref{itsa.s.}  shows that such convergence is true in the $\mathbb{P}-$a.s. sense, as well.
	Obviously $\tilde{Z}_{A,Q}$ is an $\mathbb{F}_L$--adapted process, since the space $\left(\Omega,\mathcal{F},\mathbb{P}\right)$ is complete by hypothesis, $\mathbb{F}_L$ is complete by construction and  the one dimensional OU--processes $U^n$ are $\mathbb{F}_L$--adapted. 
	
	Concerning the OU--processes, for every $x\in H$ we can express the random variables of $Z^x=\left(Z_t^x\right)_t$ as follows: 
	\[
	Z^x_{t+h}\overset{a.s.}{=}e^{hA}Z^x_t+\int_{t}^{t+h}e^{\left(t+h-s\right)A}\sqrt{Q}\,dW_{L_s}
	=e^{hA}Z^x_t+\sum_{n=1}^{\infty} \left(\int_{t}^{t+h} e^{-\lambda_n\left(t+h-s\right)}\sigma_n\,d\beta_{L_s}^n\right)e_n,\quad t,h\ge0.
	\]
	This immediately implies the Markovianity of the process, recalling the independence of the increments of the Lévy processes $\left(\beta^n_L\right)_n$. The time homogeneity is obtained by a standard argument relying on the stationarity of the increments of the same processes and the fact that the coefficients of the one--dimensional SDEs in \eqref{tea} are time--autonomous. The proof is then complete. 
\end{proof}
We close this section with an example which analyzes a common framework in applications (see, e.g., \cite{Flandoli}).
\begin{example}\label{example1}
	Let $\mathbb{T}^d=\mathbb{R}^d/\mathbb{Z}^d$ be the $d$--dimensional torus and denote by $e_k$ the functions
	\begin{equation*}
	e_k\left(x\right)\coloneqq\begin{cases}
	\cos\left(2\pi k\cdot x\right),&k\in\mathbb{Z}^d_+\\
	\sin\left(2\pi k\cdot x\right),&k\in\mathbb{Z}^d_-
	\end{cases},\quad x\in\mathbb{T}^d,
	\end{equation*}
	where $\mathbb{Z}^d_+\coloneqq\left\{\left(k_1>0\right)\text{ or }\left(k_1=0 \text{ and }k_j>0 \text{ for }j=2,\dots,d\right)\right\}$ and $\mathbb{Z}^d_-\coloneqq-\mathbb{Z}^d_+$. Then $\left\{e_k : k\in\mathbb{Z}_0^d\right\}$ constitute a complete orthonormal system for the Hilbert space  
	\[
	H=L^2_0\left(\mathbb{T}^d;\mathbb{R}\right)\coloneqq\left\{f\in L^2\left(\mathbb{T}^d;\mathbb{R}\right):\int_{\mathbb{T}^d} f\left(x\right)dx=0\right\},
	\]
	 where of course $\mathbb{Z}_0^d\coloneqq\mathbb{Z}^d\setminus\left\{0\right\}.$ In particular, for every $f\in H$, we have
	\[
	f=\sum_{k\in\mathbb{Z}^d_0}\hat{f}_ke_k,\quad \hat{f}_k\coloneqq\int_{\mathbb{T}^d}f\left(x\right)e_k\left(x\right)\,dx,\quad k\in\mathbb{Z}^d_0.   
	\]  
	We first introduce the Sobolev spaces 
	\[
	W_0^{\beta,2}\left(\mathbb{T}^d\right)\coloneqq\left\{f\in H : \sum_{k\in\mathbb{Z}^d_0}\left|k\right|^{2\beta}\hat{f}_k^2<\infty\right\},\quad \norm{f}_{W_0^{\beta,2}}^2\coloneqq \sum_{k\in\mathbb{Z}^d_0}\left|k\right|^{2\beta}\hat{f}_k^2,
	\]
	and then define the linear operator $A$ as follows:
	\[
	A\colon W_0^{2,2}\left(\mathbb{T}^d\right)\to H\quad \text{ such that }\quad Af=\Delta f=-\left(2\pi\right)^2\sum_{k\in\mathbb{Z}^d_0}\left|k\right|^2\hat{f}_ke_k.	
	\]
	In particular, the eigenvalues of $A$ corresponding to $e_k$ are $-\lambda_k=-\left(2\pi\right)^2\left|k\right|^2$, hence $A$ is unbounded and negative definite. Moreover it is selfadjoint, as well. Now we analyze Hypothesis \eqref{st_con} for two specifications of the linear, bounded, positive semidefinite operator $Q\colon H\to H$.
	\begin{itemize}
		\item Let $Q=\text{Id}$. Then $\sigma_k=1,\,k\in\mathbb{Z}^d_0$, and  \eqref{st_con} reads $$\frac{1}{\left(2\pi\right)^{2r/\alpha}}\sum_{k\in\mathbb{Z}^d_0}\frac{1}{\left|k\right|^{2r/\alpha}}<\infty\quad \text{for some }r\in\left(0,\alpha\right),$$ which is satisfied if and only if $d=1$. Hence the stochastic convolution is defined only in dimension $d=1$.%, in analogy to the Brownian case.
		\item Set $Q=Q_\eta=\left(-\Delta\right)^{-\eta}$ for $\eta>0$, the negative fractional power of the Laplacian, defined as an operator $Q_\eta\colon H\to H$ such that $$Q_\eta f=\frac{1}{\left(2\pi\right)^{2\eta}}\sum_{k\in\mathbb{Z}^d_0}\frac{1}{\left|k\right|^{2\eta}}\hat{f}_k e_k,\quad f\in H.$$
		In this case the convergence of the infinite sum in \eqref{st_con} amounts to requiring $\eta>\left(\frac{d}{2r}-\frac{1}{\alpha}\right)\vee0$. Since $r$ is chosen freely in the interval $\left(0,\alpha\right)$, Hypothesis \eqref{st_con} is satisfied if and only if
		\begin{equation}\label{nontutte}
		\eta>\left(\frac{d-2}{2\alpha}\right)\vee0.
		\end{equation} 
		This fact can be interpreted as follows: the higher the dimension $d$, the weaker the effect of the noise on the high Fourier modes needs to be in order to have the well--posedness of the stochastic convolution.
	\end{itemize}
\end{example}
	\section{Smoothing effect of the Markov Transition Semigroup}\label{sec2}
	Let us introduce the Markov transition semigroup $R=\left(R_t\right)_t$ associated with the OU--processes $\left(Z^x\right)_{x\in H}$, which is given by
	\[
	R_t\phi\left(x\right)\coloneqq\mathbb{E}\left[\phi\left(Z^x_t\right)\right],\quad x\in H,\, \phi\in \mathcal{B}_b\left(H\right),\,t\ge0,
	\]
	where $B_b\left(H\right)$ is the space of bounded, real--valued, Borel--measurable functions in $H$.	Evidently, each $R_t$ is linear and bounded from $C_b\left(H\right)$ into itself and  $R_0$ is the identity.  Our aim is to prove that, under suitable conditions, the operator $R_t$ has a smoothing effect for every $t>0$. Specifically, given a function $\phi\in B_b\left(H\right)$, in the case $\alpha\in\left(\frac{1}{2},1\right)$ we are going to show that $R_t\phi\in C^1_b\left(H\right)$ and that the following gradient estimate holds:
	\begin{equation}\label{est}
	\sup_{x\in H}\norm{\nabla R_t\phi\left(x\right)}_H\le \frac{C}{t^{{\gamma}}}\norm{\phi}_\infty\quad\text{for every $t>0$, for some }0<{\gamma}<1,\,C>0.
	\end{equation}
	\subsection{Finite dimensional case $H=\mathbb{R}^N$}
	Let $H=\mathbb{R}^N$ and 	$W^N=\left[\begin{matrix}
		\beta^1&\cdots&\beta^N
	\end{matrix}\right]^T$. We start by presenting a theorem which allows to obtain an original derivation formula for the semigroup corresponding to the finite--dimensional OU processes $Z_t^\ell\left(x\right)$. They are defined as the unique, càdlàg solutions of the linear SDEs $dZ^{\ell}_t\left(x\right)=AZ^{\ell}_t\left(x\right)dt+\sqrt{Q}\,dW^N_{\ell_t},\,Z^{\ell}_0\left(x\right)=x,$ and  can be expressed by the variation of constant formula as follows: 
	\begin{equation}\label{det_eq}
	Z^\ell_t\left(x\right)=e^{tA}x+\int_{0}^{t}e^{\left(t-s\right)A}\sqrt{Q}\,dW^N_{\ell_s},\quad t\ge0,\,\mathbb{P}-\text{a.s.},	
	\end{equation}
	where $x\in\mathbb{R}^N$ and $\ell\colon\mathbb{R}_+\to\mathbb{R}_+$ is an increasing, càdlàg function such that $\ell_0=0$ and $\ell_t>0$ for every positive $t$: the set of functions with these properties will be denoted by $\mathbb{S}$. Note that, for every $\ell\in\mathbb{S}$, $W^N_\ell=\left(W^N_{\ell_t}\right)_t$ is a càdlàg martingale with respect to the filtration $\left(\mathcal{F}^N_{\ell_t}\right)_t$, where $\left(\mathcal{F}^N_t\right)_t$ is the minimal augmented filtration generated by $W^N$. Analogously, for every $\ell\in\mathbb{S}$, we introduce the filtrations $\mathbb{F}^n_\ell=\left(\mathcal{F}^{\beta^n}_{\ell_t}\right)_t,\,n\in\mathbb{N}$, and observe that $\beta^n_\ell=\left(\beta^n_{\ell_t}\right)_t$ is a càdlàg, $\mathbb{F}^n_\ell$--martingale. The proof of such theorem is essentially based on the deterministic time--change procedure described by Zhang in \cite[Section $2$]{Z}, but exploits the linear nature of our setting to avoid the application of the \emph{Bismut--Elworthy--Li's formula} (see, e.g., \cite[Proposition $8.21$]{DP_Intro}). For the sake of completeness we report its main passages.
	\begin{theorem}\label{determ}
		 	Let $t>0,\,\phi\in C_b\left(\mathbb{R}^N\right),\,\ell\in\mathbb{S}$ and assume that $\sigma_n^2>0, \,n=1,\dots,N$.  Then the function $\mathbb{E}\left[\phi\left(Z^\ell_t\left(\cdot\right)\right)\right]$ is differentiable at any point  $x\in\mathbb{R}^N$ in every direction $h\in\mathbb{R}^N$, and 
		\begin{equation}\label{no_bel_det}
		\left\langle\nabla\mathbb{E}\left[\phi\left(Z^\ell_t\left(x\right)\right)\right],h\right\rangle
		=
		\mathbb{E}\left[\phi\left(Z_t^\ell\left(x\right)\right)\left(\sum_{n=1}^{N}\frac{1}{\sigma_n}\frac{e^{-\lambda_n t}\left\langle h,e_n\right\rangle}{\int_{0}^{t}e^{-2\lambda_n\left(t-s\right)}d\ell_s}\int_{0}^{t}e^{-\lambda_n\left(t-s\right)}d\beta_{\ell_s}^n\right)\right].
		\end{equation}
	\end{theorem}
\begin{proof}
	For every $\epsilon>0$ denote by $\ell^\epsilon\left(t\right)\coloneqq\frac{1}{\epsilon}\int_{t}^{t+\epsilon}\ell_s\,ds,\,t\ge0$, the \emph{Steklov's averages} of $\ell$. They are strictly increasing, absolutely continuous functions such that, for every $t\ge0$, $\ell^\epsilon_t\downarrow\ell_t$ as $\epsilon\downarrow 0$. Let $\gamma^\epsilon\coloneqq\left(\ell^\epsilon\right)^{-1}\colon\left[\ell_0^\epsilon,\infty\right)\to\mathbb{R}_+$ and define $Z^{\ell^\epsilon}\left(x\right)$ as in \eqref{det_eq}, i.e., for every $x\in\mathbb{R}^N$ the process $Z^{\ell^\epsilon}\left(x\right)$ is the unique solution of the linear SDE $dZ^{\ell^\epsilon}_t\left(x\right)=AZ^{\ell^\epsilon}_t\left(x\right)dt+\sqrt{Q}\,dW^N_{\ell^\epsilon_t},\,Z^{\ell^\epsilon}_0=x$. Now introduce the time--shifted processes $Y_t^{\ell^\epsilon}\left(x\right)\coloneqq Z^{\ell^\epsilon}_{\gamma^\epsilon_t}\left(x\right),\,t\ge\ell^\epsilon_0$, and observe that
	\[
		Y^{\ell^\epsilon}_t\left(x\right)=x+\int_{\ell_0^\epsilon}^tAY_s^{\ell^\epsilon}\left(x\right)\dot{\gamma^\epsilon_s}\,ds+\sqrt{Q}\left(W^N_t-W^N_{\ell_0^\epsilon}\right),\quad t\ge\ell_0^\epsilon,\,\mathbb{P}-\text{a.s.},
	\]
	which shows that $dY_t^{\ell^\epsilon}\left(x\right)=AY_t^{\ell^\epsilon}\left(x\right)\dot{\gamma^\epsilon_t}dt+\sqrt{Q}\,dW^N_{t},\,Y_{\ell_0^\epsilon}^{\ell^\epsilon}\left(x\right)=x$. Therefore, 
	\[
		Y_t^{\ell^\epsilon}\left(x\right)=e^{A\gamma_t^\epsilon}x+\int_{\ell_0^\epsilon}^te^{A\left(\gamma_t^\epsilon-\gamma_s^\epsilon\right)}\sqrt{Q}\,dW^N_s,\quad t\ge\ell_0^\epsilon,\,\mathbb{P}-\text{a.s.}
	\]
	 In particular, since $\int_{\ell_0^\epsilon}^{\ell_t^\epsilon}e^{2A\left(t-\gamma^\epsilon_s\right)}Q\,ds=
	 \int_{0}^{t}e^{2A\left(t-s\right)}Q\,d\ell^\epsilon_s$, where the integral is to be interpreted entrywise, we have
	\[
		Z_t^{\ell^\epsilon}\left(x\right)=Y_{\ell_t^\epsilon}^{\ell^\epsilon}\left(x\right)
		\sim\mathcal{N}\left(e^{At}x, \int_{0}^{t}e^{2A\left(t-s\right)}Q\,d\ell^\epsilon_s\right).
	\]
	At this point, we fix a generic $t>0,\,x\in\mathbb{R}^N$ and use \cite[Equation ($2.6$)]{Z} (it is just an application of \emph{Gronwall lemma}) to get the convergence, in the $L^2$--sense, of $Z_t^{\ell^\epsilon}\left(x\right)\to Z_t^{\ell}\left(x\right)$ as $\epsilon\downarrow0$. Moreover, recalling that  $\ell^\epsilon_t\downarrow\ell_t$ as $\epsilon\downarrow 0$, we invoke \emph{Helly's second theorem} (see \cite[Theorem $7.3$]{Natanson}) to get $\int_{0}^{t}e^{2A\left(t-s\right)}Q\,d\ell^\epsilon_s\to\int_{0}^{t}e^{2A\left(t-s\right)}Q\,d\ell_s$ as $\epsilon\downarrow 0$. Whence, 
	\begin{equation}\label{distribution_st}
	Z_t^{\ell}\left(x\right)
	\sim\mathcal{N}\left(e^{At}x, \int_{0}^{t}e^{2A\left(t-s\right)}Q\,d\ell_s\right).
	\end{equation}
	If we take $\phi \in C_b\left(\mathbb{R}^N\right)$, an explicit computation simply based on the derivation of the normal density function implies, for every direction $h\in\mathbb{R}^N$,
	\[		
		\left\langle\nabla\mathbb{E}\left[\phi\left(Z^{\ell}_t\left(x\right)\right)\right],h\right\rangle
		=
		\mathbb{E}\left[\phi\left(Z_t^{\ell}\left(x\right)\right)\left\langle\left(\int_{0}^{t}e^{2A\left(t-s\right)}Q\,d\ell_s\right)^{-1}\left(\int_{0}^{t}e^{A\left(t-s\right)}\sqrt{Q}\,dW^N_{\ell_s}\right),e^{tA}h\right\rangle\right],
	\]
	which coincides with \eqref{no_bel_det} upon expanding the notation.
\end{proof}
\begin{rem}\label{bounded}
	The previous proof does not need the continuity of the function $\phi$. Therefore, \emph{Theorem} $\ref{determ}$ holds true for every $\phi\in\mathcal{B}_b\left(\mathbb{R}^N\right)$.
\end{rem}

Now we investigate the subordinated Brownian motion case. The intuition behind the argument is to condition with respect to the $\sigma$--algebra $\mathcal{F}^L$, so that it is possible to apply the deterministic time--shift result we have just obtained in Theorem \ref{determ} upon changing the reference probability space. Let us denote by $\mathbb{W}$ the space of continuous functions from $\mathbb{R}_+$ to $\mathbb{R}^N$ vanishing at $0$ and endow it with the Borel $\sigma$--algebra $\mathcal{B}\left(\mathbb{W}\right)$ associated with the topology of locally uniform convergence. The pushforward probability measure generated by  $W^N\left(\cdot\right)\colon\left(\Omega,\mathcal{F},\mathbb{P}\right)\to\left(\mathbb{W}, \mathcal{B}\left(\mathbb{W}\right) \right)$ is denoted by $\mathbb{P}^\mathbb{W}$ and makes the canonical process $\mathfrak{x}=\left(x_t\right)_t$ a Brownian motion, where by definition
\[
	x_t\left(w\right)\coloneqq w_t,\quad w\in\mathbb{W},\,t\ge0.	
\]
We work with the usual completion $\left(\mathbb{W}, \overline{\mathcal{B}\left(\mathbb{W}\right)}, \overline{\mathbb{P}^\mathbb{W}} \right)$ of this probability space: by \cite[Theorem $7.9$]{karat}, $\mathfrak{x}$ is still a Brownian motion with respect to its minimal augmented filtration, which in turn satisfies the usual hypotheses and is denoted by $\mathbb{F}^\mathbb{W}$. In particular, note that the completeness of the space $\left(\Omega,\mathcal{F},\mathbb{P}\right)$ implies the measurability of $W^N\left(\cdot\right)\colon\left(\Omega,\mathcal{F},\mathbb{P}\right)\to\left(\mathbb{W}, \overline{\mathcal{B}\left(\mathbb{W}\right)} \right)$ and the fact that $\overline{\mathbb{P}^\mathbb{W}}$ is still the pushforward probability measure generated by $W^N\left(\cdot\right)$. Obviously, $W^N\left(\cdot\right)$ is independent from $\mathcal{F}^L$: as a consequence, a regular conditional distribution of $W^N\left(\cdot\right)$ given $\mathcal{F}^L$ is the probability kernel 
\begin{equation}\label{rcd}
	\mathbb{P}\left(W^N\!\left(\cdot\right)\in\cdot\big|\mathcal{F}^L\right)\colon \Omega\times\overline{\mathcal{B}\left(\mathbb{W}\right)}\to\left[0,1\right]\,\,\text{ such that }\,\,
	\mathbb{P}\left(W^N\!\left(\cdot\right)\in A\right)\left(w\right)\coloneqq \overline{\mathbb{P}^\mathbb{W}}\left(A\right),\, \omega\in{\Omega},\,A\in \overline{\mathcal{B}\left(\mathbb{W}\right)}.
\end{equation}
As regards the space $\mathbb{S}$, for every $t\ge0$ we introduce the map $y_t\colon\mathbb{S}\to\mathbb{R}$ defined by $y_t\left(\ell\right)\coloneqq\ell_t,\,\ell\in\mathbb{S}$, and consider the $\sigma$--algebra $\mathcal{F}^\mathbb{S}\coloneqq\sigma\left(y^{-1}_t\left(B\right),\,B\in\mathcal{B}\left(\mathbb{R}\right),\,t\ge0\right)$. Since $L\left(\cdot\right)\colon\left(\Omega,\mathcal{F}^L,\mathbb{P}\right)\to\left(\mathbb{S},\mathcal{F}^\mathbb{S}\right)$ is measurable, we can construct the pushforward probability measure $\mathbb{P}^\mathbb{S}$ on $\left(\mathbb{S},\mathcal{F}^\mathbb{S}\right)$. At this point we take into account the product space $	\left(\mathbb{W}\times \mathbb{S},\overline{\mathcal{B}\left(\mathbb{W}\right)}\otimes\mathcal{F}^\mathbb{S}, \overline{\mathbb{P}^\mathbb{W}}\otimes\mathbb{P}^\mathbb{S}\right)
$ and note that, thanks to the mutual independence of $W^N\left(\cdot\right)$ and $L\left(\cdot\right)$, the product measure $\overline{\mathbb{P}^\mathbb{W}}\otimes\mathbb{P}^\mathbb{S}$ is indeed the pushforward probability measure generated by $\psi\colon \Omega\to \mathbb{W}\times \mathbb{S},\,\psi\left(\omega\right)\coloneqq\left(W^N_\cdot\left(\omega\right), L_\cdot\left(\omega\right)\right)$. Finally, we take the process $z=\left(z_t\right)_t$ defined by 
\[
	z_t\left(w,\ell\right)\coloneqq w_{\ell_t},\quad \left(w,\ell\right)\in\mathbb{W}\times\mathbb{S},\,t\ge0,
\]
and denote by $\mathbb{F}^z=\left(\mathcal{F}^z_t\right)_t$ its natural filtration.
By construction, $W^N_{L_t}=z_t\circ \psi$ for every $t\ge0$. Putting together all these properties, we can conclude that $z$ is a Lévy process with respect to the right--continuous filtration $\mathbb{F}^z_+=\left(\mathcal{F}^z_{t+}\right)_t$, where 
\[
	\mathcal{F}^z_{t+}\coloneqq\bigcap_{\epsilon>0}\mathcal{F}^z_{t+\epsilon},\quad t\ge0.
\]
Endowing the product space with this filtration, the stochastic integral of suitable processes with respect to $z$ is well defined. Let us consider then a deterministic, continuous, bounded, $\mathbb{R}^N$--valued  process $\xi=\left(\xi_t\right)_t$:  weaker assumptions can be done on it, but in our framework these are sufficient. Clearly the subordinated Brownian motion $W^N_L$ is adapted with respect to the right--continuous filtration $\psi^{-1}\left(\mathbb{F}^z_+\right)$, therefore the usual rules of change of probability space (see, e.g., \cite[§X-$2$]{jj}) entail
\begin{equation}\label{change}
\int_0^t\xi_s\cdot dW^N_{L_s}=\left(\int_{0}^t\xi_s\cdot dz_s\right)\circ \psi,\quad	t\ge0,\,\mathbb{P}-\text{a.s.}
\end{equation}
We conclude this preliminary discussion with an important substitution formula.
\begin{lemma}\label{peggio}
	Let $\xi=\left(\xi_t\right)_t$ be a deterministic, continuous, bounded, $\mathbb{R}^N$--valued process. Then, for any $t>0$,
	\begin{equation*}
\left(	\int_{0}^{t}\xi_s\cdot dz_s\right)\left(\cdot,\ell\right)
=\int_{0}^{t}\xi_s\cdot dx_{\ell_s}\quad \overline{\mathbb{P}^\mathbb{W}}-\text{a.s., for }\mathbb{P}^\mathbb{S}-\text{a.e. }\ell\in\mathbb{S},
	\end{equation*}
	where the  integral on the right--hand side of the equality is intended in the sense of stochastic integrals by càdlàg martingales on the filtered probability space  $\left(\mathbb{W}, \overline{\mathcal{B}\left(\mathbb{W}\right)}, \overline{\mathbb{P}^\mathbb{W}};\mathbb{F}^\mathbb{W}_\ell \right).$
\end{lemma}
\begin{proof}
	Fix $t>0$ and introduce the elementary, predictable (with respect to both $\mathbb{F}^z_+$ and $\mathbb{F}^\mathbb{W}_\ell,\,\ell\in\mathbb{S}$) processes 
	\[
	\xi^m_s\coloneqq \xi_01_{\left\{0\right\}}\left(s\right)+\sum_{i=0}^{m-1}\xi_{t_i} 1_{]t_i,t_{i+1}]}\left(s\right),
	\]
	where $t_i=\frac{t}{m}i,\,i=0,\dots,m$. The continuity of $\xi$ implies that  $\xi^m\to\xi$ pointwise; furthermore, since $\xi$ is bounded, the sequence $\left(\xi^m\right)_m$ is uniformly bounded. This implies that 
	\begin{equation*}
	\int_{0}^{t}\xi_s\cdot dz_{s}=\left(\overline{\mathbb{P}^\mathbb{W}}\otimes\mathbb{P}^\mathbb{S}\right)\!-\lim_{m\to\infty}\int_{0}^{t}\xi^m_s\cdot dz_s.
	\end{equation*}
	Now  convergence in probability implies almost--sure convergence along a subsequence, hence we can say that for $\mathbb{P}^\mathbb{S}-$a.e. $\ell\in\mathbb{S}$, 
	\begin{equation}\label{su}
		\left(\int_{0}^{t}\xi^{m_k}_s\cdot dz_{s}\right)\left(\cdot,\ell\right)\underset{k\to\infty}{\longrightarrow}
		\left(\int_{0}^{t}\xi_s\cdot dz_{s}\right)\left(\cdot,\ell\right)\quad\overline{\mathbb{P}^\mathbb{W}}-\text{a.s.}
	\end{equation}
	With the same argument as above, we have
	\begin{equation}\label{su+1}
		\int_{0}^{t}\xi_s\cdot dx_{\ell_s}=\overline{\mathbb{P}^\mathbb{W}}\!-\lim_{k\to\infty}\int_{0}^{t}\xi^{m_k}_s\cdot dx_{\ell_s}\quad	\text{for every $ \ell\in\mathbb{S}$}.
	\end{equation}
On the other hand, by the very definition of stochastic integral it is immediate to notice that, for every $\left(w,\ell\right)\in\mathbb{W}\times\mathbb{S},$
\begin{equation*}
\left(\int_{0}^{t}\xi^{m_k}_s\cdot dz_s\right)\left(w,\ell\right)
=\sum_{i=0}^{{m_k}-1}\xi_{t_i}\cdot\left(z_{t_{i+1}}-z_{t_i}\right)\left(w,\ell\right)
=\sum_{i=0}^{{m_k}-1}\xi_{t_i}\cdot\left(x_{\ell_{t_{i+1}}}-x_{\ell_{t_{i}}}\right)\left(w\right)
=\left(\int_{0}^{t}\xi^{m_k}_s\cdot dx_{\ell_s}\right)\left(w\right).
\end{equation*}
Combining the last equation with \eqref{su} and \eqref{su+1} we get
\[
	\left(\int_{0}^{t}\xi_s\cdot dz_{s}\right)\left(\cdot,\ell\right)=
	\int_{0}^{t}\xi_s\cdot dx_{\ell_s}\quad \overline{\mathbb{P}^\mathbb{W}}-\text{a.s., for }\mathbb{P}^\mathbb{S}-\text{a.e. }\ell\in\mathbb{S},
\]
proving the thesis of the lemma.
\end{proof} 

A useful result due to \cite[Equation $\left(14\right)$]{Bog} shows that there exists a constant $c>0$ such that, for every $t>0$, the density $\eta_t$ of $L_t$ satisfies 
\[
\eta_t\left(s\right)\le c\,t\,s^{-1-\alpha}e^{-ts^{-\alpha}},\quad s>0.
\]
As a consequence, for every $p\ge1$ we have that ${L_t}^{-1}\in L^p$, with
\begin{equation}\label{subest}
\mathbb{E}\left[\frac{1}{L_t^p}\right]^{1/p}\le c_{\alpha,p}\frac{1}{t^{1/\alpha}}\quad \text{for some }c_{\alpha,p}>0.
\end{equation}
We are now in position to obtain the derivation formula for the Markov transition semigroup, together with an estimate on its gradient, in the subordinated Brownian motion case.
\begin{theorem}\label{detgrad}
		Let $t>0,\,\phi\in C_b\left(\mathbb{R}^N\right)$ and assume that $\sigma_n^2>0, \,n=1,\dots,N$.  Then the function $\mathbb{E}\left[\phi\left(Z^\cdot_t\right)\right]$ is differentiable at any point  $x\in\mathbb{R}^N$ in every direction $h\in\mathbb{R}^N$, and 
	\begin{equation}\label{no_bel}
	\left\langle\nabla\mathbb{E}\left[\phi\left(Z_t^x\right)\right],h\right\rangle
	=
	\mathbb{E}\left[\phi\left(Z_t^x\right)\left(\sum_{n=1}^{N}\frac{1}{\sigma_n}\frac{e^{-\lambda_n t}\left\langle h,e_n\right\rangle}{\int_{0}^{t}e^{-2\lambda_n\left(t-s\right)}dL_s}\int_{0}^{t}e^{-\lambda_n\left(t-s\right)}d\beta_{L_s}^n\right)\right].
	\end{equation}
	
	In addition, there exists $c_\alpha>0$ such that the following gradient estimate holds:
	\begin{equation}\label{est_finite}
	\sup_{x\in\mathbb{R}^N}\left|\nabla\mathbb{E}\left[\phi\left(Z_t^x\right)\right]\right|\le c_\alpha\norm{\phi}_\infty
	\sup_{n=1,\dots,N}\left(\frac{1}{\sigma_n}\sqrt[2\alpha]{\frac{2\alpha\lambda_n}{1-e^{-2\alpha\lambda_nt}}}e^{-\lambda_nt}\right)\quad \text{for every } t>0.
	\end{equation}
\end{theorem}
\begin{proof}
	Fix $t>0$ and $\phi\in C_b\left(\mathbb{R}^N\right)$. In what follows, we denote by $\mathbb{E}^\mathbb{W}\left[\cdot\right]$ the expected value of a random variable defined on $\left(\mathbb{W},\overline{\mathcal{B}\left(\mathbb{W}\right)},\overline{\mathbb{P}^\mathbb{W}}\right)$. Bearing in mind that $Z_t^x=e^{tA}x+\int_{0}^{t}e^{\left(t-s\right)A}\sqrt{Q}\,dW^N_{L_s}$,  by \eqref{change} we have, for every $x\in\mathbb{R}^N$,
	\[
		Z_t^x=\left(e^{tA}x+\int_{0}^{t}e^{\left(t-s\right)A}\sqrt{Q}\,dz_s\right)\circ \psi=\left(e^{tA}x+\int_{0}^{t}e^{\left(t-s\right)A}\sqrt{Q}\,dz_s\right)\left(W^N\left(\cdot\right),L\left(\cdot\right)\right)
	\quad \mathbb{P}-\text{a.s.}
	\]
	Therefore recalling the expression \eqref{rcd} for the regular conditional distribution $\mathbb{P}\left(W^N\!\left(\cdot\right)\in\cdot\big|\mathcal{F}^L\right)$, we apply the \emph{disintegration formula} for the conditional expectation to write 
	\begin{align*}
		\mathbb{E}\left[\phi\left(Z^x_t\right)\right]&
		=\mathbb{E}\left[\mathbb{E}\left[\phi\left(Z^x_t\right)\bigg|\mathcal{F}^L\right]\right]
		=
		\mathbb{E}\left[\int_\mathbb{W}\phi\left(\left(e^{tA}x+\int_{0}^{t}e^{\left(t-s\right)A}\sqrt{Q}\,dz_s\right)\left(w,L\left(\cdot\right)\right)\right)\overline{\mathbb{P}^\mathbb{W}}\left(dw\right)\right]\\&
		=
		\mathbb{E}\left[\restr{\mathbb{E}^\mathbb{W}\left[\phi\left(e^{tA}x+\int_{0}^{t}e^{\left(t-s\right)A}\sqrt{Q}\,dx_{\ell_s}\right)\right]}{\ell=L\left(\cdot\right)}\right]
		=\mathbb{E}\left[\restr{\mathbb{E}^\mathbb{W}\left[\phi\left(Z^\ell_t\left(x\right)\right)\right]}{\ell=L\left(\cdot\right)}\right],\quad x\in\mathbb{R}^N,
	\end{align*}
	where in the second--to--last equality we use Lemma \ref{peggio} and the fact that $\mathbb{P}^\mathbb{S}$ is the pushforward probability measure generated by $L\left(\cdot\right)$ on $\mathbb{S}$.
	Take $x\in\mathbb{R}^N$ and a direction $h\in\mathbb{R}^N$; if we can justify the derivation under the expected value, an application of \eqref{no_bel_det} immediately leads to \eqref{no_bel}, as the following computations based on the previous argument show:
	\begin{align}\label{mmm1}
		&\left\langle\nabla\mathbb{E}\left[\phi\left(Z_t^x\right)\right],h\right\rangle
		 =\mathbb{E}\left[\restr{\mathbb{E}^\mathbb{W}\left[\phi\left(Z_t^\ell\left(x\right)\right)\left(\sum_{n=1}^{N}\frac{1}{\sigma_n}\frac{e^{-\lambda_n t}\left\langle h,e_n\right\rangle}{\int_{0}^{t}e^{-2\lambda_n\left(t-s\right)}d\ell_s}\int_{0}^{t}e^{-\lambda_n\left(t-s\right)}dx_{\ell_s}^n\right)\right]}{\ell=L\left(\cdot\right)}\right]\\&
		\quad
		\!=\!\mathbb{E}\!\left[\sum_{n=1}^{N}\!\frac{1}{\sigma_n}\frac{e^{-\lambda_n t}\left\langle h,e_n\right\rangle}{\int_{0}^{t}\!e^{-2\lambda_n\left(t-s\right)}dL_s}\!\left\{\!\int_\mathbb{W}\!\left(\!\phi\!\left(\!e^{tA}x\!+\!\int_{0}^{t}\!e^{\left(t-s\right)A}\sqrt{Q}\,dz_s\right)\!\!\times\!\!\left(\int_{0}^{t}\!e^{-\lambda_n\left(t-s\right)}dz_{s}^n\right)\!\!\right)\!\left(w,L\left(\cdot\right)\right)\overline{\mathbb{P}^\mathbb{W}}\!\left(dw\right)\right\}\!\right]\notag\\&
		\quad
		\!=\mathbb{E}\left[\mathbb{E}\left[\phi\left(Z_t^x\right)\left(\sum_{n=1}^{N}\frac{1}{\sigma_n}\frac{e^{-\lambda_n t}\left\langle h,e_n\right\rangle}{\int_{0}^{t}e^{-2\lambda_n\left(t-s\right)}dL_s}\int_{0}^{t}e^{-\lambda_n\left(t-s\right)}d\beta_{L_s}^n\right)\Bigg|\mathcal{F}^L\right]\right].\notag
	\end{align}
	Indeed, such a derivation is licit, since \emph{Jensen's inequality} and  \eqref{distribution_st} entail
		\begin{align}\label{mmm2}
		\left|\restr{\mathbb{E}^\mathbb{W}\left[\phi\left(Z_t^\ell\left(x\right)\right)\left(\sum_{n=1}^{N}\frac{1}{\sigma_n}\frac{e^{-\lambda_n t}\left\langle h,e_n\right\rangle}{\int_{0}^{t}e^{-2\lambda_n\left(t-s\right)}d\ell_s}\int_{0}^{t}e^{-\lambda_n\left(t-s\right)}dx_{\ell_s}^n\right)\right]}{\ell=L\left(\cdot\right)}\right|^2
		\le\norm{\phi}_\infty^2\sum_{n=1}^N\frac{1}{\sigma_n^2}\frac{e^{-2\lambda_nt}\left|\left\langle h,e_n\right\rangle\right|^2}{\int_{0}^{t}e^{-2\lambda_n\left(t-s\right)}dL_s},
	\end{align}
with the right--hand side which does not depend on $x$ and is integrable. In fact, for every $n=1,\dots,N$, recalling that $L_1\sim \text{stable}\left(\alpha,1,\bar{c}^{1/\alpha},0\right)$ by \eqref{forg}, we have
	\[
	\int_{0}^{t}e^{-2\lambda_n\left(t-s\right)}dL_s\sim \text{stable}\left(\alpha,1,\bar{c}^{\frac{1}{\alpha}}\left(\frac{1-e^{-2\alpha\lambda_nt}}{2\alpha\lambda_n}\right)^{1/\alpha},0\right)\Longrightarrow
		\int_{0}^{t}e^{-2\lambda_n\left(t-s\right)}dL_s\sim \left(\frac{1-e^{-2\alpha\lambda_nt}}{2\alpha\lambda_n}\right)^{\frac{1}{\alpha}}L_1,
	\]
hence  by \eqref{subest} there exists $c_\alpha>0$ such that
\begin{multline}\label{+1}
	\mathbb{E}\left[\sum_{n=1}^N\frac{1}{\sigma_n^2}\frac{e^{-2\lambda_nt}\left|\left\langle h,e_n\right\rangle\right|^2}{\int_{0}^{t}e^{-2\lambda_n\left(t-s\right)}dL_s}\right]
	\le\mathbb{E}\left[\frac{1}{L_1}\right]\left(\sum_{n=1}^N\frac{e^{-2\lambda_nt}}{\sigma_n^2}\left(\frac{2\alpha\lambda_n}{1-e^{-2\alpha\lambda_nt}}\right)^{\frac{1}{\alpha}}\left|\left\langle h,e_n\right\rangle\right|^2\right)
	\\\le c_\alpha\sum_{n=1}^N\frac{e^{-2\lambda_nt}}{\sigma_n^2}\left(\frac{2\alpha\lambda_n}{1-e^{-2\alpha\lambda_nt}}\right)^{\frac{1}{\alpha}}\left|\left\langle h,e_n\right\rangle\right|^2.
\end{multline}

Concerning the gradient estimate, it is sufficient to combine \eqref{mmm1}, \eqref{mmm2} \& \eqref{+1} and to recall that the $L^1$--norm of a random variable is smaller than its $L^2$--norm to get
\begin{equation*}
\left|\left\langle\nabla\mathbb{E}\left[\phi\left(Z_t^x\right)\right],h\right\rangle\right|\le c_\alpha\norm{\phi}_\infty\sup_{n=1,\dots,N}\left(\frac{1}{\sigma_n}\sqrt[2\alpha]{\frac{2\alpha\lambda_n}{1-e^{-2\alpha\lambda_nt}}}e^{-\lambda_nt}\right)\left|h\right|,\quad x,h\in\mathbb{R}^N,
\end{equation*}
where the constant $c_\alpha$ is allowed to be different from the one in \eqref{+1}. The desired inequality \eqref{est_finite} is then recovered taking the $\sup$ for $\left|h\right|\le1$, and the proof is complete.
\end{proof}
As in Remark \ref{bounded}, note that Theorem \ref{detgrad} holds true for every $\phi\in\mathcal{B}_b\left(\mathbb{R}^N\right)$. 
\subsection{Infinite dimensional case}
In this subsection we analyze the general case where $H$ is infinite dimensional. Assuming $\sigma_n^2>0,\,n\in\mathbb{N}$, let us introduce the following Hypothesis:
\begin{equation}\label{par_Hy1}\tag{\lowerRomannumeral{2}}
	\sup_n\left(\frac{1}{\sigma_n}\sqrt[2\alpha]{\frac{2\alpha\lambda_n}{1-e^{-2\alpha\lambda_nt}}}e^{-\lambda_nt}\right)\le C_t\quad \text{ for every } t>0,\, \text{for some function }C_t>0.
\end{equation}
In this setting, for every $h\in H$ and $t>0$, we can define the real--valued random variable 
\[
\sum_{n=1}^{\infty}\frac{1}{\sigma_n}\frac{e^{-\lambda_n t}\left\langle h,e_n\right\rangle}{\int_{0}^{t}e^{-2\lambda_n\left(t-s\right)}dL_s}\int_{0}^{t}e^{-\lambda_n\left(t-s\right)}d\beta_{L_s}^n
\coloneqq L^2-\lim_{N\to\infty}\left(\sum_{n=1}^{N}\frac{1}{\sigma_n}\frac{e^{-\lambda_n t}\left\langle h,e_n\right\rangle}{\int_{0}^{t}e^{-2\lambda_n\left(t-s\right)}dL_s}\int_{0}^{t}e^{-\lambda_n\left(t-s\right)}d\beta_{L_s}^n\right).
\]
Indeed, with the same argument as the one in \eqref{+1},  Hypothesis \eqref{par_Hy1} yields
\begin{align*}
\mathbb{E}\left[\left|\sum_{n=m}^M\frac{1}{\sigma_n}\frac{e^{-\lambda_n t}\left\langle h,e_n\right\rangle}{\int_{0}^{t}e^{-2\lambda_n\left(t-s\right)}dL_s}\int_{0}^{t}e^{-\lambda_n\left(t-s\right)}d\beta_{L_s}^n\right|^2\right]
&\le c_\alpha \sum_{n=m}^M\frac{e^{-2\lambda_nt}}{\sigma_n^2}\left(\frac{2\alpha\lambda_n}{1-e^{-2\alpha\lambda_nt}}\right)^{\frac{1}{\alpha}}\left|\left\langle h,e_n\right\rangle\right|^2\\
&\le c_{\alpha}\,C_t^2 \,\left(\sum_{n=m}^M\left|\left\langle h,e_n\right\rangle\right|^2\right)\underset{m,M\to\infty}{\longrightarrow}0,
\end{align*}
where $c_\alpha>0$. In particular,
\begin{equation}\label{var}
\mathbb{E} \left[\left|\sum_{n=1}^{\infty}\frac{1}{\sigma_n}\frac{e^{-\lambda_n t}\left\langle h,e_n\right\rangle}{\int_{0}^{t}e^{-2\lambda_n\left(t-s\right)}dL_s}\int_{0}^{t}e^{-\lambda_n\left(t-s\right)}d\beta_{L_s}^n\right|^2\right]^{\frac{1}{2}}\le \sqrt{c_\alpha}\,C_t\,\norm{h}_H.
\end{equation}
Hence the following, useful property holds:
\begin{equation}\label{property}
\sum_{n=1}^{\infty}\frac{1}{\sigma_n}\frac{e^{-\lambda_n t}\left\langle h_m,e_n\right\rangle}{\int_{0}^{t}e^{-2\lambda_n\left(t-s\right)}dL_s}\int_{0}^{t}e^{-\lambda_n\left(t-s\right)}d\beta_{L_s}^n\overset{L^2}{\longrightarrow}
\sum_{n=1}^{\infty}\frac{1}{\sigma_n}\frac{e^{-\lambda_n t}\left\langle h,e_n\right\rangle}{\int_{0}^{t}e^{-2\lambda_n\left(t-s\right)}dL_s}\int_{0}^{t}e^{-\lambda_n\left(t-s\right)}d\beta_{L_s}^n \quad \text{as }h_m\to h.
\end{equation}
At this point we can present the main theorem of the paper.
\begin{theorem}\label{main}
	 Assume $\sigma_n^2>0,\,n\in\mathbb{N}$, together with Hypotheses \eqref{st_con}\&\,\eqref{par_Hy1}. 
	 
	 Then for every $\phi\in B_b\left(H\right)$ and $t>0$ the function $R_t\phi\in C^1_b\left(H\right)$ and there exists $c_\alpha>0$ such that
	 \begin{equation}\label{vabecedo}
	 		\sup_{x\in H}\norm{\nabla R_t\phi\left(x\right)}_H\le c_\alpha \,C_t\norm{\phi}_\infty\quad \text{for every }t>0.
	 \end{equation}
	  
	 Moreover, given $\phi\in C_b\left(H\right)$ and $t>0$, for every $x,h\in H$ the Gateaux derivative of $R_t\phi$ at $x$ along the direction $h$ is given by
	\begin{equation}\label{no_BEL}
	\left\langle\nabla R_t\phi\left(x\right),h\right\rangle
	=\mathbb{E}\left[\phi\left(Z_t^x\right)\left(\sum_{n=1}^{\infty}\frac{1}{\sigma_n}\frac{e^{-\lambda_n t}\left\langle h,e_n\right\rangle}{\int_{0}^{t}e^{-2\lambda_n\left(t-s\right)}dL_s}\int_{0}^{t}e^{-\lambda_n\left(t-s\right)}d\beta_{L_s}^n\right)\right].
	\end{equation}
\end{theorem}
\begin{proof}
Fix $t>0$ and a function $\phi\in C_b\left(H\right)$. 

We first consider the case $\dim H=N,$ identifying $H$ with $\mathbb{R}^N$, as usual. Evidently \eqref{no_BEL} coincides with \eqref{no_bel} and the map $x\mapsto\nabla R_t\left(x\right)$ is a continuous function from $\mathbb{R}^N$ into itself: this follows from dominated convergence, together with $\phi\in C_b\left(\mathbb{R}^N\right)$ and $Z_t^{x_n}\to Z_t^{x}$ a.s. as $x_n\to x$. Moreover, Hypothesis \eqref{par_Hy1} applied to \eqref{est_finite} directly entails \eqref{vabecedo}, 
therefore $R_t\phi\in C_b^1\left(\mathbb{R}^N\right)$. In order to pass to infinite dimension it is convenient to write
\begin{align}\label{trick}
R_t\phi\left(x+h\right)-R_t\phi\left(x\right)&=\int_{0}^{1} \left\langle\nabla R_t\phi\Big(\left(1-\rho \right)x+\rho \left(x+h\right)\Big)\,,\,h\right\rangle d\rho\notag\\&
=
\int_{0}^{1}\mathbb{E}\left[\phi\left(Z_t^{x+\rho h}\right)\left(\sum_{n=1}^{N}\frac{1}{\sigma_n}\frac{e^{-\lambda_n t}\left\langle h,e_n\right\rangle}{\int_{0}^{t}e^{-2\lambda_n\left(t-s\right)}dL_s}\int_{0}^{t}e^{-\lambda_n\left(t-s\right)}d\beta_{L_s}^n\right)\right]d\rho.
\end{align}

We now consider the general case $\dim H=\infty$. Let $\pi_N$ be the projection onto the first $N$ Fourier components and $H_N$ be its range. Due to the diagonal structure of our model, the projections $\pi_NZ_t^x$ of the OU--process are, $\mathbb{P}-$a.s.,
\[
\pi_N Z^x_t=\sum_{n=1}^Ne^{-\lambda_nt}\left\langle x,e_n\right\rangle e_n+\sum_{n=1}^N\left(\int_{0}^{t}e^{-\lambda_n\left(t-s\right)}\sigma_nd\beta_{L_s}^n\right)e_n, \quad N\in\mathbb{N}.
\]
Therefore introducing the operators $A_N\coloneqq\restr{A}{H_N}$ and $Q_N\coloneqq\restr{Q}{H_N}$, which map $H_N$ into itself, we can write $\pi_NZ^x_t=e^{tA_N}\left(\pi_Nx\right)+\tilde{Z}_{A_N,Q_N}\left(t\right)$: this shows that such projections are OU--processes in $H_N$. Thus, the dominated convergence theorem together with the expression in \eqref{trick} and the continuity of $\phi$ give
\begin{align*}
R_t\phi\left(x+h\right)-R_t\phi\left(x\right)&=\lim_{N\to\infty}\mathbb{E}\left[\phi\left(\pi_NZ_t^{x+h}\right)-\phi\left(\pi_NZ_t^{x}\right)\right]\\
&=\lim_{N\to\infty}\int_{0}^{1}\mathbb{E}\left[\phi\left(\pi_NZ_t^{x+\rho h}\right)\left(\sum_{n=1}^{N}\frac{1}{\sigma_n}\frac{e^{-\lambda_n t}\left\langle h,e_n\right\rangle}{\int_{0}^{t}e^{-2\lambda_n\left(t-s\right)}dL_s}\int_{0}^{t}e^{-\lambda_n\left(t-s\right)}d\beta_{L_s}^n\right)\right]d\rho
\\&=
\int_{0}^1\mathbb{E}\left[\phi\left(Z_t^{x+\rho h}\right)\left(\sum_{n=1}^{\infty}\frac{1}{\sigma_n}\frac{e^{-\lambda_n t}\left\langle h,e_n\right\rangle}{\int_{0}^{t}e^{-2\lambda_n\left(t-s\right)}dL_s}\int_{0}^{t}e^{-\lambda_n\left(t-s\right)}d\beta_{L_s}^n\right)\right]d\rho.
\end{align*}
Now we can define $D_{t,x}\left(h\right)\coloneqq \mathbb{E}\left[\phi\left(Z_t^x\right)\left(\sum_{n=1}^{\infty}\frac{1}{\sigma_n}\frac{e^{-\lambda_n t}\left\langle h,e_n\right\rangle}{\int_{0}^{t}e^{-2\lambda_n\left(t-s\right)}dL_s}\int_{0}^{t}e^{-\lambda_n\left(t-s\right)}d\beta_{L_s}^n\right)\right]$: it is the Fr\'echet differential of $R_t\phi$ at $x$ (hence, in particular, \eqref{no_BEL} is verified). To see this, it is sufficient to note that the linear operator $D_{t,x}\left(\cdot\right)$ is continuous by the property in \eqref{property} and to apply H\"older's inequality, the dominated convergence theorem and \eqref{var} to get, for a positive constant $c_\alpha$,
\begin{align*}
\left|R_t\phi\left(x+h\right)-R_t\phi\left(x\right)-D_{t,x}\left(h\right)\right|\le 
c_{\alpha}\,C_t\norm{h}_H\int_{0}^{1}\mathbb{E}\left[\left|\phi\left(Z_t^{x+\rho h}\right)-\phi\left(Z_t^x\right)\right|^2\right]^{1/2}d\rho
=o\left(\norm{h}_H\right).
\end{align*}
The upper bound \eqref{vabecedo} for the norm of the gradient is then obtained by \eqref{var} from the next, straightforward computation:
\begin{equation*}
\norm{\nabla R_t\phi\left(x\right)}_H=\sup_{\norm{h}_H \le 1}\left|\left\langle\nabla R_t\phi\left(x\right),h\right\rangle\right|=\sup_{\norm{h}_H \le 1}\left|D_{t,x}\left(h\right)\right|\le c_{\alpha}\,C_t\,\norm{\phi}_\infty,\quad x\in H.
\end{equation*}
 We also note that
\[
\sup_{\norm{h}_H\le1}\left|\left(D_{t,x_n}-D_{t,x}\right)\left(h\right)\right|\le c_{\alpha}\,C_t\,\mathbb{E}\left[\left|\phi\left(Z_t^{x_n}\right)-\phi\left(Z_t^x\right)\right|^2\right]^{1/2}\to0 \quad \text{as }x_n\to x:
\] 
this proves the continuity of the map $x\mapsto D_{t,\cdot}$, hence $R_t\phi\in C^1_b\left(H\right)$.

Finally, we need to study the case where $\phi$ is just Borel measurable and bounded, without the hypothesis of continuity. In order to do this, it is sufficient to observe that by the \emph{mean value theorem} and \eqref{vabecedo} we have, for every $\phi\in C_b^2\left(H\right),$
\begin{equation}\label{lip}
	\left|R_t\phi\left(x\right)-R_t\phi\left(y\right)\right|\le c_{\alpha}\,C_t\,\norm{\phi}_\infty\norm{x-y}_H,\quad x,y\in H.
\end{equation}
Being $R_t$ Markovian, \cite[Lemma $7.1.5$]{DP1} implies that the same holds true for every $\phi\in B_b\left(H\right)$. In particular, $R_t$ maps bounded, Borel measurable functions in bounded, Lip--continuous functions. The semigroup law let us write $R_t\phi=R_s\left(R_{t-s}\phi\right)$ for some $0<s<t$, which proves $R_t\phi\in C_b^1\left(H\right)$. The bound \eqref{vabecedo} follows from \eqref{lip}, hence the proof is complete.
\end{proof}
We now focus on the gradient estimate \eqref{est}. We need to substitute Hypothesis \eqref{par_Hy1} with the following, stronger one:
\begin{equation}\label{par_Hy}\tag{\lowerRomannumeral{3}}
\sup_n\left(\frac{1}{\sigma_n}\sqrt[2\alpha]{\frac{2\alpha\lambda_n}{1-e^{-2\alpha\lambda_nt}}}e^{-\lambda_nt}\right)\le C_0\frac{1}{t^\gamma},\quad \text{for every } t>0,\, \text{for some }C_0>0,\,0<\gamma<1.
\end{equation}
In other terms, in Hypothesis \eqref{par_Hy1} we take $C_t\coloneqq C_0\,t^{-\gamma},\,t>0$, for some $C_0>0,\,\gamma\in\left(0,1\right)$. 
\begin{rem}\label{rem_condition}
	Observe that, for every $n\in\mathbb{N}$, the term 
	\[
	\frac{1}{\sigma_n}\sqrt[2\alpha]{\frac{2\alpha\lambda_n}{1-e^{-2\alpha\lambda_nt}}}e^{-\lambda_nt}\sim\frac{1}{\sigma_n}\frac{1}{t^{1/\left(2\alpha\right)}} \quad\text{as }t\downarrow0.
	\]
	Therefore, Hypothesis \eqref{par_Hy} should be verified only in the case $\alpha\in\left(\frac{1}{2},1\right)$ and for some $\gamma\in \left[\frac{1}{2\alpha},1\right)$. 
	
	It is also worth noticing that Hypothesis \eqref{par_Hy} is equivalent to the next condition:
	\begin{equation}\label{par_Hy_eq}\tag{\lowerRomannumeral{3}$'$}
	\sigma_n\ge C_1\,\lambda_n^{\frac{1}{2\alpha}-\gamma},\quad n\in\mathbb{N},
	\end{equation}
	for some $C_1>0$ and $\gamma\in\left[\frac{1}{2\alpha},1\right)$.
	A short argument proving the latter fact is shown in \cite[\emph{Hypothesis~(N)}]{PZ}.
\end{rem}
At this point the next result is immediate.
\begin{corollary}\label{corol}
	Consider $\alpha\in\left(\frac{1}{2},1\right)$ and assume $\sigma_n^2>0,\,n\in\mathbb{N}$, together with Hypotheses \eqref{st_con}\&\,\eqref{par_Hy}. 
	
	Then for every $\phi\in\mathcal{B}_b\left(H\right)$ the function $R_t\phi\in C_b^1\left(H\right),\,t>0,$ and the gradient estimate \eqref{est} holds, namely there exists a constant $C>0$ such that
	\[
		\sup_{x\in H}\norm{\nabla R_t\phi\left(x\right)}_H\le \frac{C}{t^{\gamma}}\norm{\phi}_\infty\quad \text{for every }t>0,
	\]
	where $\gamma\in\left[\frac{1}{2\alpha},1\right)$ is the one appearing in Hypothesis \eqref{par_Hy}. 
\end{corollary}
\begin{example}\label{example2}
	We investigate Hypothesis \eqref{par_Hy}, in its equivalent formulation \eqref{par_Hy_eq} provided by Remark \ref{rem_condition}, in the same framework as in Example \ref{example1}. So we take $A=\Delta$ (hence $-\lambda_k=-\left(2\pi\right)^2\left|k\right|^2,\,k\in\mathbb{Z}_0^d$) and study two possible choices for $Q$.
	\begin{itemize}
		\item If $Q=\text{Id}$, then 
		\[
		1\ge \frac{1}{\left(2\pi\left|k\right|\right)^{2\left(\gamma-\frac{1}{2\alpha}\right)}},\quad k\in\mathbb{Z}_0^d
		\] 
		for every $\gamma\in\left[\frac{1}{2\alpha},1\right)$. Therefore, in dimension $d=1$ both conditions \eqref{st_con} and \eqref{par_Hy} are satisfied. In particular, motivated by the fact that $R_t$ is a regularization operator with $R_0= \text{Id}$, we are interested in the behavior of $\nabla R_t\phi$ around $0$, where $\phi\in \mathcal{B}_b\left(H\right)$. Therefore we choose $\gamma=\frac{1}{2\alpha}$ and Corollary \ref{corol} provides the next estimate:
		\[
			\sup_{x\in H}\norm{\nabla R_t\phi\left(x\right)}_H\le C\,\frac{1}{t^{2\alpha}}\norm{\phi}_\infty\quad \text{for every } t>0,
		\]
		for a positive constant $C$.
		\item If $Q=Q_\eta=\left(-\Delta\right)^{-\eta}$ for $\eta>0$, then $\sigma^{\left(\eta\right)}_k=\lambda_k^{-\eta/2},\,k\in\mathbb{Z}_0^d$, and \eqref{par_Hy_eq} holds true if and only if $\eta\le2\gamma-\frac{1}{\alpha}$. Since we can take any $\gamma\in\left[\frac{1}{2\alpha},1\right)$, the aforementioned condition holds as soon as $\eta<2-\frac{1}{\alpha}.$ Combining this result with \eqref{nontutte} obtained in Example \ref{example1}, we conclude that Hypotheses \eqref{st_con} and \eqref{par_Hy} simultaneously hold if and only if $$\eta\in\left(\max\left\{\frac{d-2}{2\alpha},0\right\}\!, \,2
		-\frac{1}{\alpha}\right).$$ 
		It then follows that there exist negative fractional powers of the Laplacian $Q_\eta=\left(-\Delta\right)^{-\eta}$ meeting the requirements of Corollary \ref{corol}  up to dimension $d=3$. Specifically, for $d=1,2$ there is a $Q_\eta$ with the searched properties for every $\alpha\in\left(\frac{1}{2},1\right)$, whereas in dimension $d=3$ we can find such a $Q_\eta$ only for $\alpha\in\left(\frac{3}{4},1\right)$.
	\end{itemize}
\end{example}
\section*{Acknowledgment}
	I thank Professor Franco Flandoli for useful discussions and valuable insight into the subject.


\begin{thebibliography}{99}
	
	\bibitem{BB}
	{\sc Benth, F. E., Benth, J. S., \& Koekebakker, S.} (2008). \emph{Stochastic modelling of electricity and related markets} (Vol. 11). World Scientific.
	
	\bibitem{Bog}
	{\sc Bogdan, K., Stós, A., \& Sztonyk, P.} (2003). Harnack inequality for stable processes on d-sets. \emph{Studia Math, 158}(2), 163--198.
	\bibitem{DP_Intro}
{\sc Da Prato, G.} (2014). \emph{Introduction to stochastic analysis and Malliavin calculus} (Vol. 13). Springer.
	
	\bibitem{DP1}
	{\sc Da Prato, G., \& Zabczyk, J.} (1996). \emph{Ergodicity for infinite dimensional systems} (Vol. 229). Cambridge University Press.
	
	\bibitem{DP2}
	{\sc Da Prato, G., \& Zabczyk, J.} (2014). \emph{Stochastic equations in infinite dimensions}. Cambridge university press.
	
	\bibitem{Flandoli}
	{\sc Flandoli, F., \& Luo, D.} (2020). Convergence of transport noise to Ornstein–Uhlenbeck for 2D Euler equations under the enstrophy measure. \emph{The Annals of Probability, 48}(1), 264--295.
	
	\bibitem{jj}
	{\sc Jacod, J.} (2006). \emph{Calcul stochastique et problemes de martingales} (Vol. 714). Springer.
	
	\bibitem{karat}
	{\sc Karatzas, I., \& Shreve, S. E.} (1998). Brownian motion. In \emph{Brownian Motion and Stochastic Calculus} (pp. 47--127). Springer, New York, NY.
	
	\bibitem{KU}
	{\sc Kusuoka, S.} (2010). Malliavin calculus for stochastic differential equations driven by subordinated Brownian motions. \emph{Kyoto Journal of Mathematics, 50}(3), 491--520.
	
	\bibitem{Natanson}
	{\sc Natanson, I. P.} (2016).\emph{ Theory of functions of a real variable}. Courier Dover Publications.
	
	\bibitem{PZ}
	{\sc Priola, E., \& Zabczyk, J.} (2011). Structural properties of semilinear SPDEs driven by cylindrical stable processes. \emph{Probability theory and related fields, 149}(1-2), 97--137.
	
\bibitem{Protter}	
	{\sc Protter, P. E.} (2005). \emph{Stochastic integration and differential equations}. Springer, Berlin, Heidelberg.
	
\bibitem{Sato}	
	{\sc Ken-Iti, S.} (1999). \emph{Lévy processes and infinitely divisible distributions}. Cambridge university press.
	
\bibitem{Wink}
{\sc Winkelbauer, A.} (2012). Moments and absolute moments of the normal distribution. \emph{arXiv preprint arXiv:1209.4340}

\bibitem{Z}
{\sc Zhang, X.} (2013). Derivative formulas and gradient estimates for SDEs driven by $\alpha$-stable processes. \emph{Stochastic Processes and their Applications, 123}(4), 1213--1228.
\end{thebibliography}
\end{document}